\documentclass[journal,onecolumn,12pt]{article}

\usepackage{amssymb}
\usepackage{caption2}
\usepackage{amsmath}
\usepackage{multirow}
\usepackage{amsthm}
\usepackage{graphicx}
\usepackage{CJK}
\usepackage{enumerate}
\usepackage{appendix}
\usepackage{bm}
\usepackage{tikz}
\usetikzlibrary{positioning,chains,fit,shapes,calc}

\newtheorem{theorem}{Theorem}[section]
\newtheorem{lemma}[theorem]{Lemma}
\newtheorem{proposition}[theorem]{Proposition}
\newtheorem{corollary}[theorem]{Corollary}

\newtheorem{remark}[theorem]{Remark}

\begin{document}
\title{On the lower bound for kissing numbers of $\ell_p$-spheres in high dimensions}

\author{Chengfei Xie\thanks{C. Xie is with the School of Mathematical Sciences, Capital Normal University, Beijing 100048, China (email: cfxie@cnu.edu.cn).}
~and Gennian Ge\thanks{Corresponding author. G. Ge is with the School of Mathematical Sciences, Capital Normal University, Beijing 100048, China (e-mail: gnge@zju.edu.cn). The research of G. Ge is supported by the National Key Research and Development Program of China under Grant Nos. 2020YFA0712100  and  2018YFA0704703, National Natural Science Foundation of China under Grant No. 11971325, and Beijing Scholars Program.}}

\maketitle

\begin{abstract}
In this paper, we give some new lower bounds for the kissing number of $\ell_p$-spheres. These results improve the previous work due to Xu (2007). Our method is based on coding theory.
\smallskip
\end{abstract}
\medskip


\smallskip


\section{Introduction}
Let $S^{n-1}$ be the unit sphere in $\mathbb{R}^n$. The (translative) kissing number problem asks the maximum number of nonoverlapping translates $S^{n-1}+\bm{x}$ that can touch $S^{n-1}$ at its boundary. This is an old and difficult problem in discrete geometry. The exact answer is only known in dimensions $1, 2, 3, 4, 8$, and $24$. In dimensions $1$ and $2$, the problem is trivial; in dimension $3$, the problem is known as the Gregory-Newton Problem and was solved by Sch{\"u}tte and van der Waerden \cite{schutte1952problem} (see also \cite{MR76369} for another proof); in dimension $4$, the problem was solved by Musin \cite{MR2415397} via an extension of Delsarte's method; in dimensions $8$ and $24$, the problem was solved by Leven\v{s}te\u{\i}n \cite{levenshtein1979bounds} and Odlyzko and Sloane \cite{MR530296} independently.

Let $K_2(n)$ be the kissing number of $S^{n-1}$. The best upper bound for $K_2(n)$ in high dimensions is due to Kabatjanski\u{\i} and Leven\v{s}te\u{\i}n \cite{MR0514023}: $K_2(n)\leq2^{0.401n(1+o(1))}$. Using a sphere covering argument, Shannon \cite{MR103137} and Wyner \cite{MR180417} obtained a lower bound $K_2(n)\geq c\sqrt n(2/\sqrt3)^n$. Recently, Jenssen et al$.$ \cite{MR3836667} improved the lower bound by a linear factor in the dimension. See also {Fern{\'a}ndez} et al$.$ \cite{2021arXiv211101255G} for constant factor improvement.

In this paper, we consider the kissing number of $\ell_p$-spheres. For $p\geq1$, let $S_{p}^{n-1}(R)$ be the $\ell_p$-sphere with radius $R$ and centered at $\bm{0}$ in $\mathbb{R}^n$, that is, $S_{p}^{n-1}(R):=\left\{\bm{x}\in\mathbb{R}^n: \|\bm{x}\|_p=R\right\}$,
where the $\ell_p$-norm $\|\cdot\|_p$ is defined by $\|\bm{x}\|_p=\left(\sum_{i=1}^n|x_i|^p\right)^{1/p}$ for $\bm{x}=(x_1, x_2, \ldots, x_n)$. We simply write $S_p^{n-1}=S_{p}^{n-1}(1)$. Let $K_p(n)$ be the kissing number of $S_p^{n-1}$. Minkowski-Hadwiger theorem \cite{MR91490} implies an upper bound $K_p(n)\leq3^n-1$. This bound was improved by Sah et al$.$ \cite{MR4064778} for $p\geq2$. Much less is known about the upper bound when $p$ is between $1$ and $2$.

On the lower bound, Larman and Zong \cite{MR1668102} proved that $K_p(n)\geq(9/8)^{n(1+o(1))}=2^{0.1699n(1+o(1))}$. Xu \cite{MR2301531} improved this result for every $p\geq1$, for instance, $K_3(n)\geq2^{0.4564n(1+o(1))}$. Our main result is an improvement to the work of Xu. Since our result does not have an explicit formula, we list some numerical results here:
$$
K_1(n)\geq2^{0.1247n(1+o(1))}+2^{0.1825n(1+o(1))}+2^{0.1554n(1+o(1))}+\cdots;
$$
$$
K_2(n)\geq2^{0.2059n(1+o(1))}+2^{0.1381n(1+o(1))}+2^{0.0584n(1+o(1))}+\cdots;
$$
$$
K_3(n)\geq cn2^{0.4564n(1+o(1))}+2^{0.1562n(1+o(1))}+2^{0.0425n(1+o(1))}+\cdots.
$$

We give some explanation to our results. In the lower bound for $K_2(n)$, the $2^{0.2059n(1+o(1))}$ term is the same as the lower bound due to Xu, so we improve the lower bound by adding the remainder terms $2^{0.1381n(1+o(1))}+2^{0.0584n(1+o(1))}+\cdots$. In the lower bound for $K_3(n)$, the $2^{0.4564n(1+o(1))}$ term is the same as the lower bound due to Xu, so we improve the leading term by a factor of $n$ and add some remainder terms.

Our idea comes from coding theory. The translative kissing number $K_p(n)$ is equal to the largest size of an \textit{$\ell_p$-spherical code} with minimum distance $1$ (see Lemma \ref{qiumianma}). We choose a discrete set $X$ from $S_p^{n-1}$. Applying ideas from coding theory, we are able to find a large subset of $X$, in which points have pairwise distance larger than or equal to $1$. This gives a lower bound for  $K_p(n)$.

\section{An improved Gilbert-Varshamov type bound}
Let $A_p(n, d)$ be the maximum size of a subset of  $S_p^{n-1}$ in which the points have pairwise $\ell_p$-distance at least $2d$; that is,
$$
A_p(n, d):=\max\{|C|:C\subseteq S_p^{n-1}\text{ and }d_p(\bm{x}, \bm{y})\geq2d, \forall\bm{x}, \bm{y}\in C\},
$$
where $d_p(\bm{x}, \bm{y}):=\|\bm{x}-\bm{y}\|_p$ is the $\ell_p$-distance between $\bm{x}$ and $\bm{y}$. In other words, $A_p(n, d)$ is the largest size of an $\ell_p$-spherical code with minimum distance $2d$. The following lemma is an easy observation.
\begin{lemma}\label{qiumianma}
The translative kissing number $K_p(n)$ of  $S_p^{n-1}$ is equal to $A_p(n, 1/2)$.
\end{lemma}
\begin{proof}
For convenience, let $k_1=K_p(n)$ and $k_2=A_p(n, 1/2)$.

Suppose $S_p^{n-1}, S_p^{n-1}+\bm{x}_1, S_p^{n-1}+\bm{x}_2, \ldots, S_p^{n-1}+\bm{x}_{k_1}$ form a kissing configuration. For every $i$, if $d_p(\bm{0}, \bm{x}_i)>2$, then $S_p^{n-1}+\bm{x}_i$ and   $S_p^{n-1}$ do not share a common point; if $d_p(\bm{0}, \bm{x}_i)<2$, then $S_p^{n-1}+\bm{x}_i$ and   $S_p^{n-1}$ are overlapping. Thus, $d_p(\bm{0}, \bm{x}_i)=2$ and $\frac{1}{2}\bm{x}_i\in S_p^{n-1}$ for every $i$. Moreover, $d_p(\bm{x}_i, \bm{x}_j)\geq2$ for $i\neq j$. So $d_p(\frac{1}{2}\bm{x}_i, \frac{1}{2}\bm{x}_j)\geq1$ for $i\neq j$. Therefore, $\{\frac{1}{2}\bm{x}_1, \frac{1}{2}\bm{x}_2, \ldots, \frac{1}{2}\bm{x}_{k_1}\}$ is an $\ell_p$-spherical code with minimum distance $1$, i.e. $k_2\geq k_1$.

On the other hand, suppose $\{\bm{x}_1, \bm{x}_2, \ldots, \bm{x}_{k_2}\}$ is an $\ell_p$-spherical code with minimum distance $1$. Then $S_p^{n-1}+2\bm{x}_1, S_p^{n-1}+2\bm{x}_2, \ldots, S_p^{n-1}+2\bm{x}_{k_2}$ are nonoverlapping, and $S_p^{n-1}+2\bm{x}_i$ touches $S_p^{n-1}$ at $\bm{x}_i$ for every $i$. So $k_1\geq k_2$. Thus the lemma follows.
\end{proof}

For a positive integer $m\leq n$, which will be determined later, we define a family $\mathcal{J}(m, n)$ of subsets of $\mathbb{R}^n$ recursively. Define $m_1:=m$ and
$$
J_1(m, n):=\left\{\bm{u}=(u_1, u_2, \ldots, u_n)\in\{0, \pm1\}^n:\sum_{i=1}^n|u_i|^p=m\right\}.
$$
Suppose we have defined $m_i$ and $J_i(m, n)$. Then we define
\begin{equation}\label{midaxiao}
m_{i+1}:=\left\lfloor m_i/2^p\right\rfloor
\end{equation}
and
$$
J_{i+1}(m, n):=\left\{\bm{u}=(u_1, u_2, \ldots, u_n)\in\{0, \pm(m/m_{i+1})^{1/p}\}^n:\sum_{i=1}^n|u_i|^p=m\right\}.
$$
This process terminates when $m_r<2^p$ for some $r$. So we obtain $\{m_1>m_2>\ldots>m_r\}$ and $\mathcal{J}(m, n)=\{J_1(m, n), J_2(m, n), \ldots, J_r(m, n)\}$. And we have the following proposition.

\begin{proposition}
For $\mathcal{J}(m, n)$ defined above, the following statements hold.
\begin{enumerate}
  \item If $i\neq j$, then $J_i(m, n)\cap J_j(m, n)=\emptyset$.
  \item For every $1\leq i\leq r$ and for every $\bm{u}\in J_i(m, n)$, $\bm{u}$ has exactly $n-m_i$ zero coordinates.
  \item For every $1\leq i\leq r$,
\begin{equation}\label{jidaxiao}
|J_i(m, n)|={n \choose m_i}2^{m_i}.
\end{equation}
  \item For every $1\leq i\leq r$ and for every $\bm{u}\in J_i(m, n)$, the $\ell_p$-norm of $\bm{u}$ is $m^{1/p}$.
  \item If $i\neq j$, then for every $\bm{u}\in J_i(m, n)$ and $\bm{v}\in J_j(m, n)$, $d_p(\bm{u}, \bm{v})\geq m^{1/p}$.
\end{enumerate}
\end{proposition}
\begin{proof}
The first four statements are trivial.

In order to prove the last statement, let $\bm{u}=(u_1, u_2, \ldots, u_n)\in J_i(m, n)$ and $\bm{v}=(v_1, v_2, \ldots, v_n)\in J_j(m, n)$, where $1\leq i<j\leq r$. Without loss of generality, assume that $u_1=u_2=\cdots=u_{m_i}=(m/m_{i})^{1/p}$ and $u_{m_i+1}=u_{m_i+2}=\cdots=u_{n}=0$. In other words, $\bm{u}=(m/m_{i})^{1/p}\cdot1^{m_i}0^{n-m_i}$. For $1\leq k\leq m_i$, we have $v_k\in \{0, \pm(m/m_{j})^{1/p}\}$, and
\begin{equation*}
\begin{split}
|u_k-v_k|^p&\geq\min\left\{|(m/m_{i})^{1/p}-0|^p, |(m/m_{i})^{1/p}-(m/m_{j})^{1/p}|^p, |(m/m_{i})^{1/p}+(m/m_{j})^{1/p}|^p\right\}\\
&=\min\left\{|(m/m_{i})^{1/p}-0|^p, |(m/m_{i})^{1/p}-(m/m_{j})^{1/p}|^p\right\}\\
&=\min\left\{\frac{m}{m_{i}}, \frac{m}{m_{i}}\cdot|1-(m_i/m_{j})^{1/p}|^p\right\}\\
&=\frac{m}{m_{i}}\min\left\{1, |1-(m_i/m_{j})^{1/p}|^p\right\}.
\end{split}
\end{equation*}
Since $j>i$, it follows that $m_j\leq m_{i+1}=\lfloor\frac{m_i}{2^p}\rfloor\leq\frac{m_i}{2^p}$. Thus $m_i/m_j\geq2^p$, and
$$
|u_k-v_k|^p\geq\frac{m}{m_{i}}\min\left\{1, |1-(m_i/m_{j})^{1/p}|^p\right\}\geq\frac{m}{m_{i}}\min\left\{1, |1-(2^p)^{1/p}|^p\right\}=\frac{m}{m_{i}}.
$$
Therefore,
$$
d_p(\bm{u}, \bm{v})^p=\sum_{k=1}^n|u_k-v_k|^p\geq\sum_{k=1}^{m_i}|u_k-v_k|^p\geq\sum_{k=1}^{m_i}\frac{m}{m_{i}}=m.
$$
This completes the proof.
\end{proof}
For every $i$, let $J'_i(m, n)$ be a largest subset of $J_i(m, n)$ with the property that $d_p(\bm{u}, \bm{v})\geq m^{1/p}$ for every $\bm{u}, \bm{v}\in J'_i(m, n)$. Since we have proved that $d_p(\bm{u}, \bm{v})\geq m^{1/p}$ if $\bm{u}\in J'_i(m, n)\subseteq J_i(m, n)$ and $\bm{v}\in J'_j(m, n)\subseteq J_j(m, n)$ for $i\neq j$,  the set
$$
\frac{1}{m^{1/p}}\bigcup_{i=1}^r J'_i(m, n):=\left\{\bm{x}\in\mathbb{R}^n:m^{1/p}\bm{x}\in\bigcup_{i=1}^r J'_i(m, n)\right\}
$$
is an $\ell_p$-spherical code with minimum distance $1$. So
\begin{equation}\label{apji}
A_p(n, 1/2)\geq\left|\frac{1}{m^{1/p}}\bigcup_{i=1}^r J'_i(m, n)\right|=\left|\bigcup_{i=1}^r J'_i(m, n)\right|=\sum_{i=1}^r\left|J'_i(m, n)\right|.
\end{equation}

For $1\leq i\leq r$ and $\bm{u}\in J_i(m, n)$, define
$$
B_{i, n}(\bm{u}, m):=\left\{\bm{v}\in J_i(m, n): d_p(\bm{u}, \bm{v})<m^{1/p}\right\},
$$
which is the open $\ell_p$-ball centered at $\bm{u}$ with radius $m^{1/p}$ in the metric space $(J_i(m, n), \|\cdot\|_p)$. Note that the size of $B_{i, n}(\bm{u}, m)$ is independent of $\bm{u}$. If we write $B_{i, n}(m)$ for the size of $B_{i, n}(\bm{u}, m)$, then
\begin{equation}\label{bidaxiao}
B_{i, n}(m)=\sum_{2t+2^px<m_i}{m_i\choose t}{n-m_i\choose t}{m_i-t\choose x}2^t.
\end{equation}
Using the above notations, we have the following theorem, which is a Gilbert-Varshamov type bound for $\left|J'_i(m, n)\right|$.
\begin{theorem}\label{gvjie}
For every $1\leq i\leq r$, we have
\begin{equation}\label{jixiajie}
\left|J'_i(m, n)\right|\geq\left\lceil\frac{\left|J_i(m, n)\right|}{B_{i, n}(m)}\right\rceil=\left\lceil\frac{{n \choose m_i}2^{m_i}}{B_{i, n}(m)}\right\rceil.
\end{equation}
\end{theorem}
The following corollary is immediate and it is our main result.
\begin{corollary}\label{gaijin}
\begin{equation}\label{apxiajie}
A_p(n, 1/2)\geq\max_{1\leq m\leq n}\sum_{i=1}^r\left\lceil\frac{{n \choose m_i}2^{m_i}}{B_{i, n}(m)}\right\rceil.
\end{equation}
\end{corollary}
\begin{remark}
In \cite[Lemma 2.1]{MR2301531}, the lower bound for $A_p(n, 1/2)$ is given by $\max_{1\leq m\leq n}\left\lceil\frac{{n \choose m_1}2^{m_1}}{B_{1, n}(m)}\right\rceil$. So Corollary \ref{gaijin} gives an improvement.
\end{remark}
\begin{proof}[Proof of Theorem \ref{gvjie}]
Let $i$ be given and $J=\left\lceil\frac{\left|J_i(m, n)\right|}{B_{i, n}(m)}\right\rceil$. We choose points from $J_i(m, n)$ recursively. At first, we arbitrarily choose $\bm{u}_1$ in $J_i(m, n)$. Suppose we have chosen $\bm{u}_1, \bm{u}_2, \ldots, \bm{u}_k$ for some $k<J$. The set
$$
J_i(m, n)\setminus\left(\bigcup_{j=1}^kB_{i, n}(\bm{u}_j, m)\right)
$$
has size at least
$$
\left|J_i(m, n)\right|-\sum_{j=1}^k\left|B_{i, n}(\bm{u}_j, m)\right|=\left|J_i(m, n)\right|-kB_{i, n}(m)>0.
$$
So we can choose $\bm{u}_{k+1}$ from $J_i(m, n)\setminus\left(\bigcup_{j=1}^kB_{i, n}(\bm{u}_j, m)\right)$ and $d_p(\bm{u}_{k+1}, \bm{u}_j)\geq m^{1/p}$ for every $1\leq j\leq k$. This process continues as long as $k<J$. Therefore, $\{\bm{u}_1, \bm{u}_2, \ldots, \bm{u}_J\}$ is a subset of $J_i(m, n)$, in which the points have pairwise distance at least $m^{1/p}$. And hence $\left|J'_i(m, n)\right|\geq J$.
\end{proof}
\section{Some numerical results for small $p$}
It seems that there does not exist an explicit formula for the lower bound in Corollary \ref{gaijin}. So we give some numerical results for small $p$ in this section. In \cite{MR2301531}, Xu gives the lower bound for $\max_{1\leq m\leq n}\left\lceil\frac{{n \choose m_1}2^{m_1}}{B_{1, n}(m)}\right\rceil$. We still need to estimate the rest terms in right hand side of inequality (\ref{apxiajie}).

Define
$$
F_p(\sigma)=\frac{{n \choose \lfloor\sigma n\rfloor}2^{\lfloor\sigma n\rfloor}}{\sum_{2t+2^px<\lfloor\sigma n\rfloor}{\lfloor\sigma n\rfloor\choose t}{n-\lfloor\sigma n\rfloor\choose t}{\lfloor\sigma n\rfloor-t\choose x}2^t}, \sigma\in(0,1).
$$
Then by equations (\ref{midaxiao})-(\ref{bidaxiao}) and inequality (\ref{apxiajie}), we have
\begin{equation*}
A_p(n, 1/2)\geq\max_{0<\sigma<1}\sum_{i=1}^rF_p\left(\frac{\sigma}{2^{(i-1)p}}\right).
\end{equation*}
\subsection{The value of $r$}
We first estimate the value of $r$. Suppose $m=\lceil2^{kp}+2^{(k-1)p}+\cdots+2^p\rceil$ for some $k$. Then
$$
m_1=m=\lceil2^{kp}+2^{(k-1)p}+\cdots+2^p\rceil\in\left[2^{kp}+2^{(k-1)p}+\cdots+2^p, 2^{kp}+2^{(k-1)p}+\cdots+2^p+1\right].
$$
We calculate
\begin{equation*}
\begin{split}
m_2=\left\lfloor\frac{m_1}{2^p}\right\rfloor&\in\left[\lfloor2^{(k-1)p}+2^{(k-2)p}+\cdots+1\rfloor, \lfloor2^{(k-1)p}+2^{(k-2)p}+\cdots+1+2^{-p}\rfloor\right]\\
&\subseteq\left[2^{(k-1)p}+2^{(k-2)p}+\cdots+2^p, 2^{(k-1)p}+2^{(k-2)p}+\cdots+1+2^{-p}\right],
\end{split}
\end{equation*}
and
\begin{equation*}
\begin{split}
m_3=\left\lfloor\frac{m_2}{2^p}\right\rfloor&\in\left[\lfloor2^{(k-2)p}+2^{(k-3)p}+\cdots+1\rfloor, \lfloor2^{(k-2)p}+2^{(k-3)p}+\cdots+2^{-p}+2^{-2p}\rfloor\right]\\
&\subseteq\left[2^{(k-2)p}+2^{(k-3)p}+\cdots+2^p, 2^{(k-2)p}+2^{(k-3)p}+\cdots+2^{-p}+2^{-2p}\right].
\end{split}
\end{equation*}
So
$$
m_k\in\left[2^p, 2^p+1+2^{-p}+\cdots+2^{-(k-1)p}\right],
$$
and
$$
m_{k+1}\in\left[1, 1+2^{-p}+2^{-2p}+\cdots+2^{-kp}\right]\subseteq[1, 2).
$$
Therefore $m_{k+1}=1$ and $r=k+1$ if $m=\lceil2^{kp}+2^{(k-1)p}+\cdots+2^p\rceil$. Note that $\lceil2^{kp}+2^{(k-1)p}+\cdots+2^p\rceil\in[2^{kp}, 2^{(k+1)p})$. On the other hand, if $m\in[2^{kp}, \lceil2^{kp}+2^{(k-1)p}+\cdots+2^p\rceil)$, then $m_k$ may be less than $2^p$. So we conclude that $r=\lfloor\log_{2^p}m\rfloor+1$ or $r=\lfloor\log_{2^p}m\rfloor$.

\subsection{The behavior of $F_p(\sigma)$}
In this subsection, we investigate the asymptotic behavior of $F_p(\sigma)$.

Let $H(\sigma)$ be the entropy function defined as
$$
H(\sigma)=\left\{
  \begin{array}{ll}
    0, & \hbox{if $\sigma=0$ or $\sigma=1$;} \\
    -\sigma\log_2\sigma-(1-\sigma)\log_2(1-\sigma), & \hbox{if $0<\sigma<1$.}
  \end{array}
\right.
$$
We have the following theorem.
\begin{theorem}[\cite{MR2301531}]
We have
$$
\lim_{n\rightarrow\infty}\frac{1}{n}\log_2F_p(\sigma)\geq\min_{0\leq y\leq\min\{\sigma/2, 1-\sigma\}}f_p(\sigma, y),
$$
where
$$
f_p(\sigma, y)=(\sigma-y)\left(1-H\left(\frac{\sigma-2y}{2^p(\sigma-y)}\right)\right)+H(\sigma)-\sigma H\left(\frac y\sigma\right)-(1-\sigma)H\left(\frac{y}{1-\sigma}\right).
$$
\end{theorem}

\subsection{Numerical results for some special values of $p$}
Let $g_p(\sigma)=\min_{0\leq y\leq\min\{\sigma/2, 1-\sigma\}}f_p(\sigma, y)$. We list some numerical results for special values of $p$.
\\

For $p=1$, see left hand side of Figure \ref{g1g2} for the graph of $g_1(\sigma)$. $g_1(\sigma)$ attains its maximum $0.1825$ at $\sigma_0=0.2605$.  So
\begin{equation*}
\begin{split}
A_1(n, 1/2)&\geq\max_{0\leq \sigma\leq 1}\sum_{i=1}^rF_1\left(\frac{\sigma}{2^{i-1}}\right)\\
&\geq\sum_{i=1}^rF_1\left(\frac{2\sigma_0}{2^{i-1}}\right)\\
&\geq F_1\left(2\sigma_0\right)+F_1\left(\sigma_0\right)+F_1\left(\frac{\sigma_0}{2}\right)+\cdots\\
&\geq 2^{g_1(2\sigma_0)\cdot n(1+o(1))}+2^{g_1(\sigma_0)\cdot n(1+o(1))}+2^{g_1(\sigma_0/2)\cdot n(1+o(1))}+\cdots\\
&=2^{0.1247n(1+o(1))}+2^{0.1825n(1+o(1))}+2^{0.1554n(1+o(1))}+\cdots.
\end{split}
\end{equation*}
Although $2^{0.1247n(1+o(1))}+2^{0.1554n(1+o(1))}+\cdots=o(2^{0.1825n(1+o(1))})$, we still write them explicitly since they improve the previous bound.
\begin{remark}
In \cite{MR1767027}, Talata obtained $A_1(n, 1/2)\geq2^{0.1825n(1+o(1))}$ as well.
\end{remark}

For $p=2$, see right hand side of Figure \ref{g1g2} for the graph of $g_2(\sigma)$. $g_2(\sigma)$ attains its maximum $0.2059$ at $\sigma_0=0.3881$.  So
\begin{equation*}
\begin{split}
A_2(n, 1/2)&\geq\max_{0\leq \sigma\leq 1}\sum_{i=1}^rF_2\left(\frac{\sigma}{2^{2(i-1)}}\right)\\
&\geq\sum_{i=1}^rF_2\left(\frac{\sigma_0}{4^{i-1}}\right)\\
&\geq F_2\left(\sigma_0\right)+F_2\left(\frac{\sigma_0}{4}\right)+F_2\left(\frac{\sigma_0}{4^{2}}\right)+\cdots\\
&\geq 2^{g_2(\sigma_0)\cdot n(1+o(1))}+2^{g_2(\sigma_0/4)\cdot n(1+o(1))}+2^{g_2(\sigma_0/16)\cdot n(1+o(1))}+\cdots\\
&=2^{0.2059n(1+o(1))}+2^{0.1381n(1+o(1))}+2^{0.0584n(1+o(1))}+\cdots.
\end{split}
\end{equation*}
 \begin{figure}
\centering
{
\includegraphics[width=7.5cm]{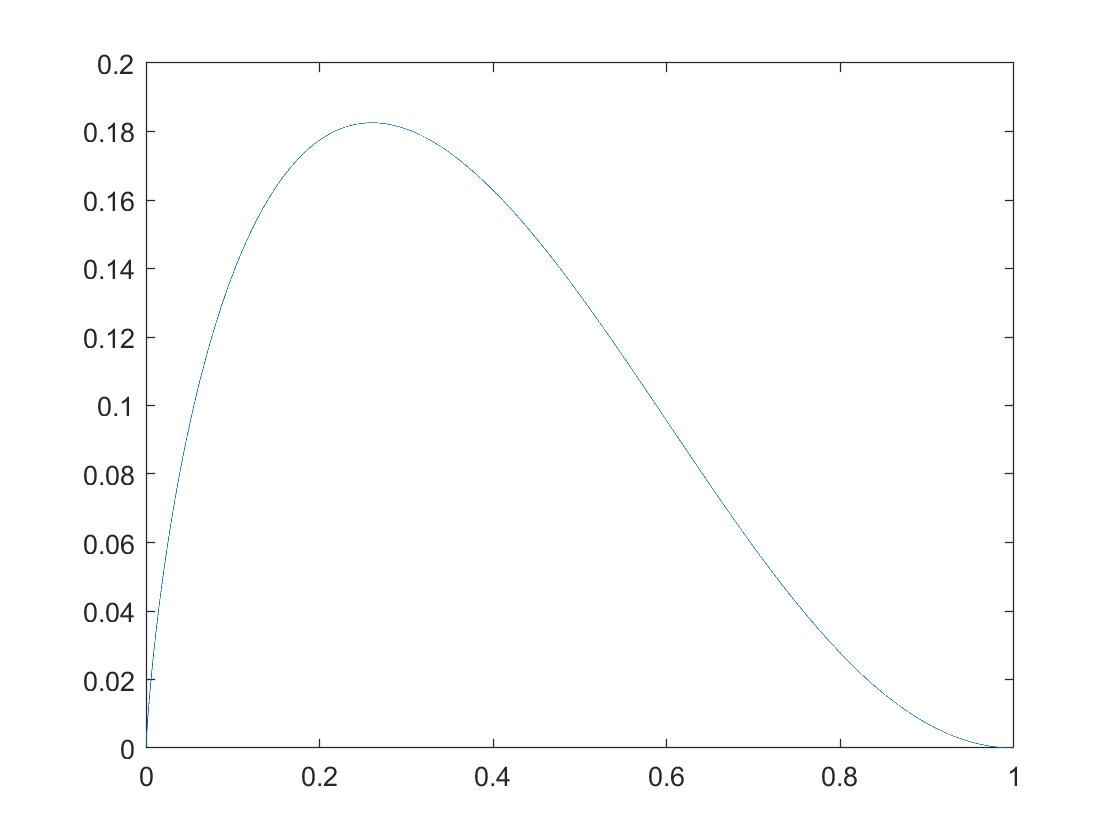}}
\hspace{0.5in}
{
\includegraphics[width=7.5cm]{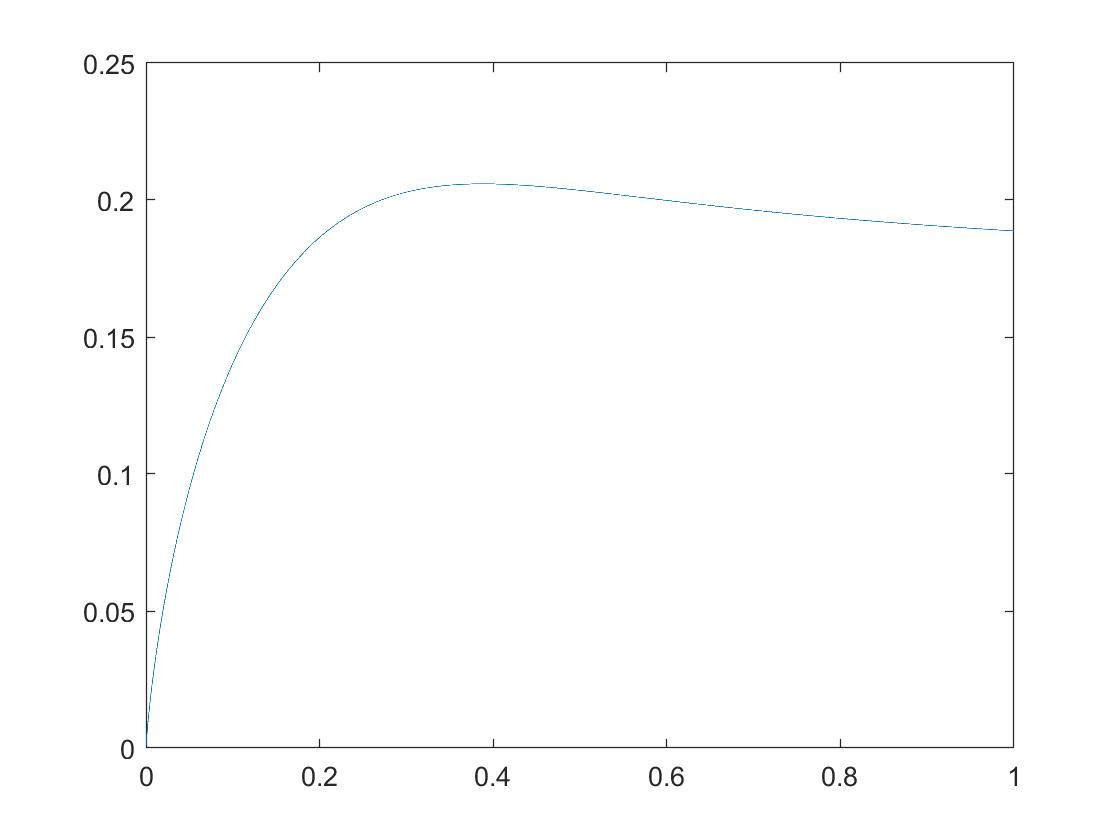}}
\hspace{0.5in}
\caption{The graphs of $g_1(\sigma)$ and $g_2(\sigma)$}\label{g1g2}
\end{figure}
We also write the  $2^{0.1381n(1+o(1))}+2^{0.0584n(1+o(1))}+\cdots=o(2^{0.2059n(1+o(1))})$ terms explicitly.
\\

For $p=2.1$, see Figure \ref{g21} for the graph of $g_{2.1}(\sigma)$.
 \begin{figure}
\centering
{
\includegraphics[width=9cm]{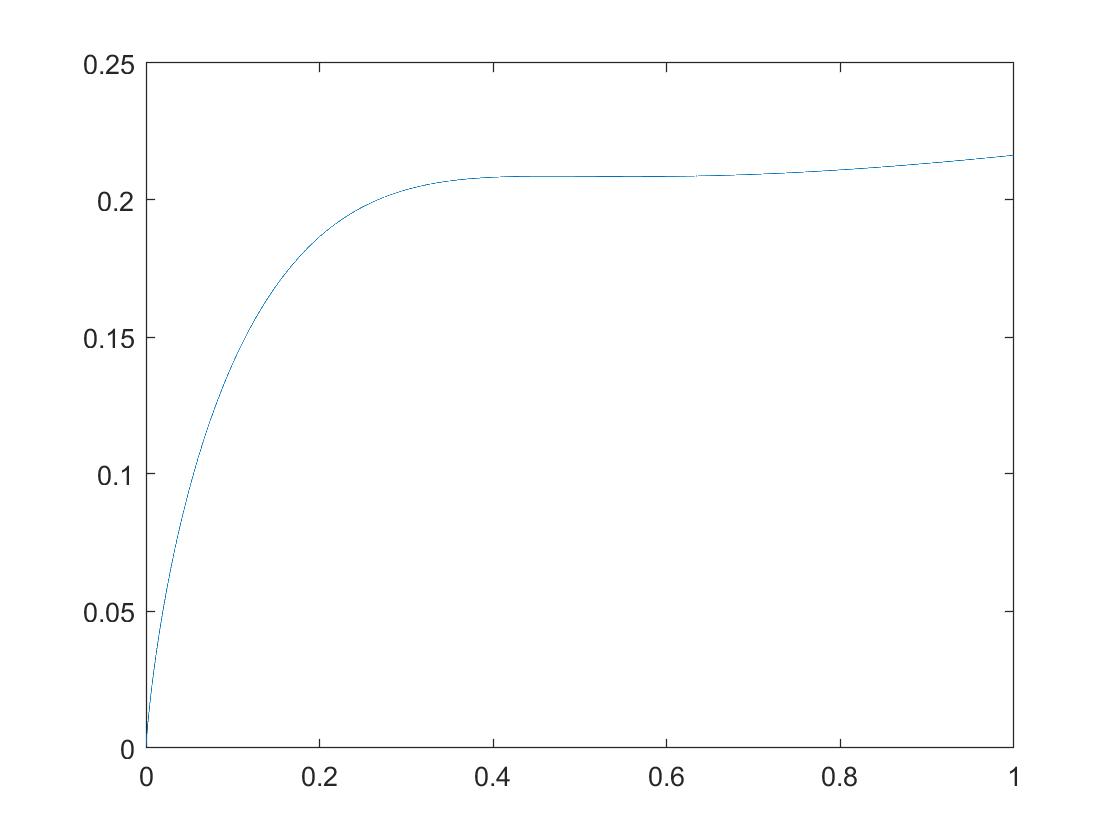}}

\caption{The graph of $g_{2.1}(\sigma)$}\label{g21}
\end{figure}
$g_{2.1}(\sigma)$ attains its maximum $0.2163$ at $\sigma_0=0.9998$.  So
\begin{equation*}
\begin{split}
A_{2.1}(n, 1/2)&\geq\max_{0\leq \sigma\leq 1}\sum_{i=1}^rF_{2.1}\left(\frac{\sigma}{2^{2.1(i-1)}}\right)\\
&\geq\sum_{i=1}^rF_{2.1}\left(\frac{\sigma_0}{4.2871^{i-1}}\right)\\
&\geq F_{2.1}\left(\sigma_0\right)+F_{2.1}\left(\frac{\sigma_0}{4.2871}\right)+F_{2.1}\left(\frac{\sigma_0}{4.2871^{2}}\right)+\cdots\\
&\geq 2^{g_{2.1}(\sigma_0)\cdot n(1+o(1))}+2^{g_{2.1}(\sigma_0/4.2871)\cdot n(1+o(1))}+2^{g_{2.1}(\sigma_0/18.3792)\cdot n(1+o(1))}+\cdots\\
&=2^{0.2163n(1+o(1))}+2^{0.1944n(1+o(1))}+2^{0.0995n(1+o(1))}+\cdots.
\end{split}
\end{equation*}
We also write the  $2^{0.1944n(1+o(1))}+2^{0.0995n(1+o(1))}+\cdots=o(2^{0.2163n(1+o(1))})$ terms explicitly.

\section{Some numerical results for large $p$}
There exists a threshold $p_0\approx2.1$ (we do not attempt to calculate the exact value of $p_0$) such that when $p>p_0$, $F_p(\sigma)$ attains its maximum at $\sigma=1$. For $\sigma=1$, i.e. $m=n$, we have another lower bound. Let $m=n$, and recall inequalities (\ref{apji}) and (\ref{jixiajie}). We have
\begin{equation*}
\begin{split}
A_p(n, 1/2)&\geq\sum_{i=1}^r\left|J'_i(n, n)\right|\\
&=\left|J'_1(n, n)\right|+\sum_{i=2}^r\left|J'_i(n, n)\right|\\
&\geq\left|J'_1(n, n)\right|+\sum_{i=2}^r\left\lceil\frac{{n \choose m_i}2^{m_i}}{B_{i, n}(n)}\right\rceil\\
&=\left|J'_1(n, n)\right|+\sum_{i=2}^rF_p\left(\frac{1}{2^{p(i-1)}}\right).
\end{split}
\end{equation*}
Indeed, we can improve the lower bound for $\left|J'_1(n, n)\right|$ slightly.

\subsection{An improvement of the lower bound for $\left|J'_1(n, n)\right|$}
Recall the definition of $J_1(n, n)$ and $J'_1(n, n)$. $J_1(n, n)=\{\pm1\}^n$ and $J'_1(n, n)$ is a largest subset of $\{\pm1\}^n$ in which points have pairwise distance larger than or equal to $n^{1/p}$. For $\bm{u}, \bm{v}\in\{\pm1\}^n$, let $d_H(\bm{u}, \bm{v}):=|\{i:u_i\neq v_i\}|$ be the Hamming distance between them. The following lemma is an easy observation.
\begin{lemma}
For every $\bm{u}, \bm{v}\in\{\pm1\}^n$, we have
$$
\left(d_p(\bm{u}, \bm{v})\right)^p=2^p\cdot d_H(\bm{u}, \bm{v}).
$$
\end{lemma}
By this lemma, it suffices to find a largest subset of $\{\pm1\}^n$, in which points have pairwise Hamming distance larger than or equal to $\lceil n/2^p\rceil$. Recall the definition of $B_{1, n}(\bm{u}, n)$ and we have
\begin{equation*}
\begin{split}
B_{1, n}(\bm{u}, n)&=\left\{\bm{v}\in \{\pm1\}^n: d_p(\bm{u}, \bm{v})<n^{1/p}\right\}\\
&=\left\{\bm{v}\in \{\pm1\}^n: 2^p\cdot d_H(\bm{u}, \bm{v})<n\right\}\\
&=\left\{\bm{v}\in \{\pm1\}^n: d_H(\bm{u}, \bm{v})\leq\lceil n/2^p\rceil-1\right\}.
\end{split}
\end{equation*}
So $B_{1, n}(n)=|B_{1, n}(\bm{u}, n)|=\sum_{k=0}^{\lceil n/2^p\rceil-1}{n\choose k}$. We have the following theorem, which gives a  better lower bound for $|J'_1(n, n)|$ than that in inequality (\ref{jixiajie}).
\begin{theorem}[\cite{MR2096836}]
There exists a positive constant $c$ such that
$$
|J'_1(n, n)|\geq c\frac{2^n}{B_{1, n}(n)}\log_2B_{1, n}(n).
$$
\end{theorem}
Note that
$$
\lim_{n\rightarrow\infty}\frac{1}{n}\log_2B_{1, n}(n)=H\left(\frac{1}{2^p}\right),
$$
by Stirling's formula. So
$$
|J'_1(n, n)|\geq c\frac{n2^n}{B_{1, n}(n)}=cn2^{n(1-H(2^{-p})+o(1))},
$$
for some constant $c$ (maybe depends on $p$). Although $n=2^{o(n)}$, we write it explicitly to represent the improvement.
\subsection{Numerical results for some special values of $p$}
As before, let $g_p(\sigma)=\min_{0\leq y\leq\min\{\sigma/2, 1-\sigma\}}f_p(\sigma, y)$. We list some numerical results for special values of $p$.
\\

For $p=2.2$, see left hand side of Figure \ref{g22g3} for the graph of $g_{2.2}(\sigma)$.
We have
\begin{equation*}
\begin{split}
A_{2.2}(n, 1/2)&\geq\left|J'_1(n, n)\right|+\sum_{i=2}^rF_{2.2}\left(\frac{1}{2^{2.2(i-1)}}\right)\\
&\geq cn2^{n(1-H(2^{-2.2})+o(1))}+F_{2.2}\left(0.2176\right)+F_{2.2}\left(0.0474\right)+\cdots\\
&\geq cn2^{n(1-H(2^{-2.2})+o(1))}+2^{g_{2.2}(0.2176)\cdot n(1+o(1))}+2^{g_{2.2}(0.0474)\cdot n(1+o(1))}+\cdots\\
&=cn2^{0.2442n(1+o(1))}+2^{0.1913n(1+o(1))}+2^{0.0915n(1+o(1))}+\cdots.
\end{split}
\end{equation*}
\\

For $p=3$, see right hand side of Figure \ref{g22g3} for the graph of $g_3(\sigma)$. We have
\begin{equation*}
\begin{split}
A_3(n, 1/2)&\geq\left|J'_1(n, n)\right|+\sum_{i=2}^rF_{3}\left(\frac{1}{2^{3(i-1)}}\right)\\
&\geq cn2^{n(1-H(2^{-3})+o(1))}+F_{3}\left(0.1250\right)+F_3\left(0.0156\right)+\cdots\\
&\geq cn2^{n(1-H(2^{-3})+o(1))}+2^{g_{3}(0.1250)\cdot n(1+o(1))}+2^{g_{3}(0.0156)\cdot n(1+o(1))}+\cdots\\
&=cn2^{0.4564n(1+o(1))}+2^{0.1562n(1+o(1))}+2^{0.0425n(1+o(1))}+\cdots.
\end{split}
\end{equation*}
\begin{figure}
\centering
{
\includegraphics[width=7.5cm]{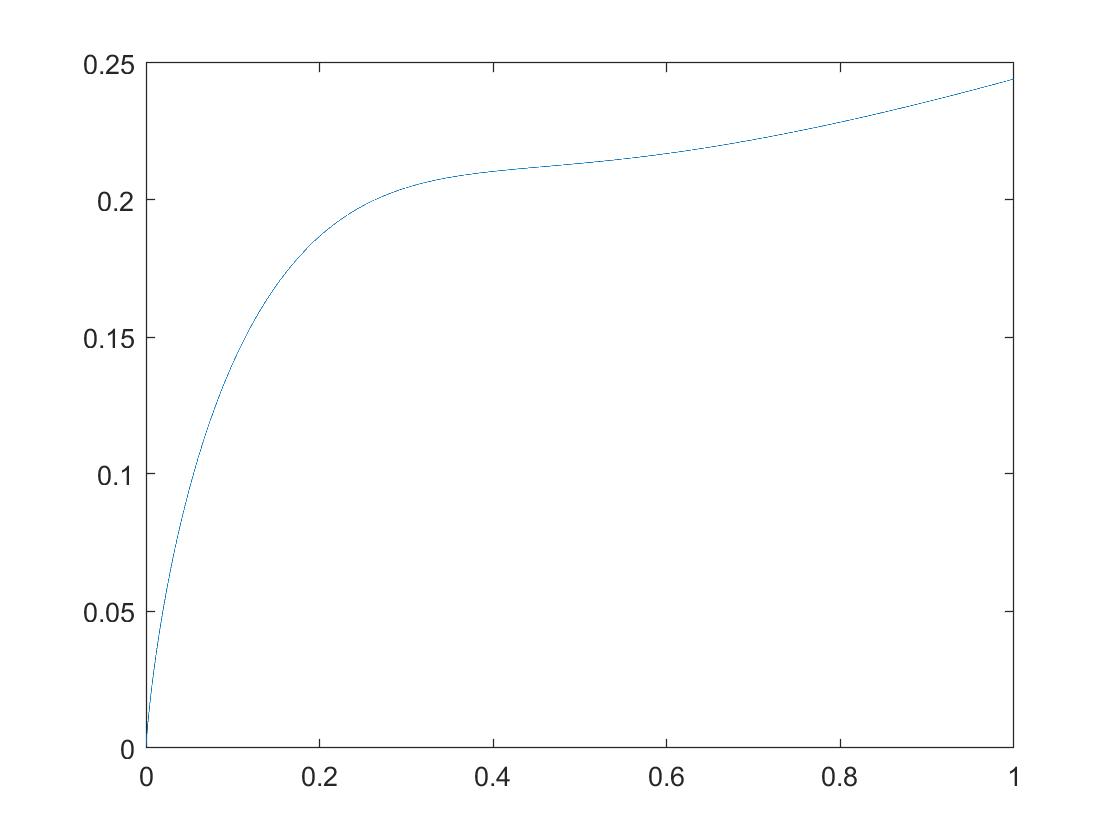}}
\hspace{0.5in}
{
\includegraphics[width=7.5cm]{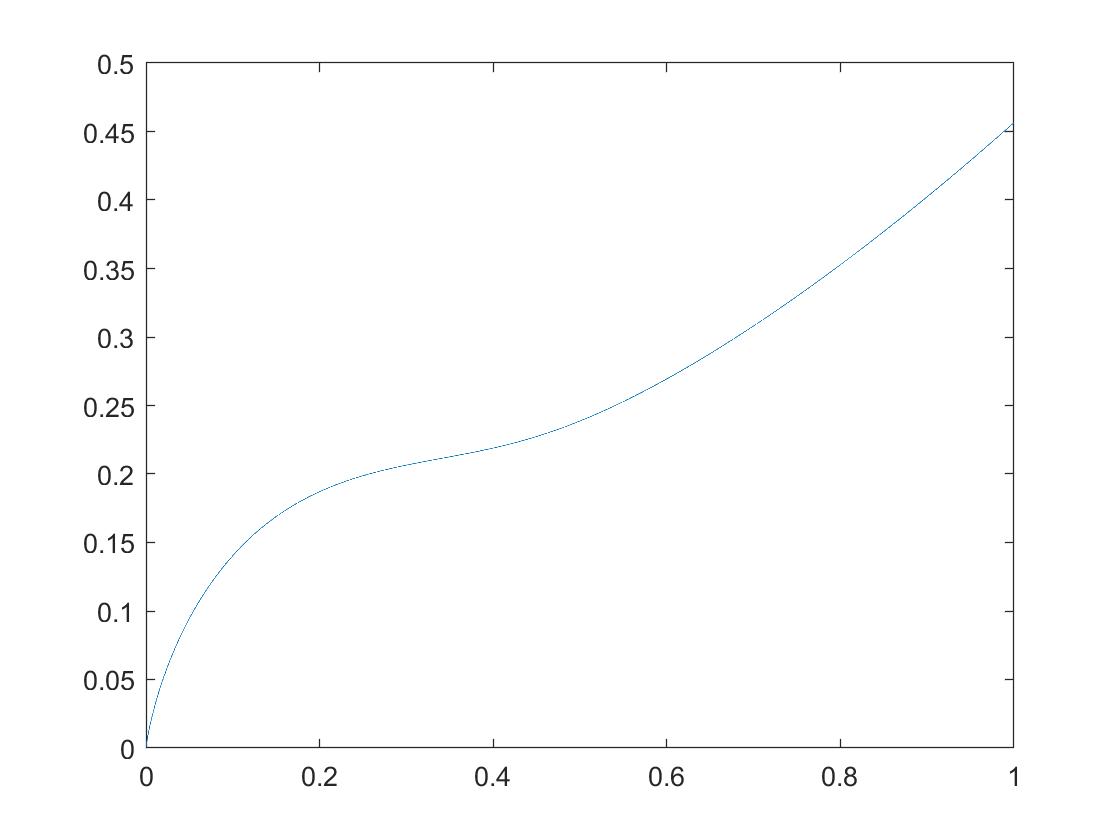}}
\hspace{0.5in}
\caption{The graphs of $g_{2.2}(\sigma)$ and $g_3(\sigma)$}\label{g22g3}
\end{figure}
\\

For $p=4$, see Figure \ref{g4} for the graph of $g_4(\sigma)$.
We have
\begin{equation*}
\begin{split}
A_4(n, 1/2)&\geq\left|J'_1(n, n)\right|+\sum_{i=2}^rF_{4}\left(\frac{1}{2^{4(i-1)}}\right)\\
&\geq cn2^{n(1-H(2^{-4})+o(1))}+F_{4}\left(0.0625\right)+F_4\left(0.0039\right)+\cdots\\
&\geq cn2^{n(1-H(2^{-4})+o(1))}+2^{g_{4}(0.0625)\cdot n(1+o(1))}+2^{g_{4}(0.0039)\cdot n(1+o(1))}+\cdots\\
&=cn2^{0.6627n(1+o(1))}+2^{0.1083n(1+o(1))}+2^{0.0145n(1+o(1))}+\cdots.
\end{split}
\end{equation*}
 \begin{figure}
\centering
{
\includegraphics[width=9cm]{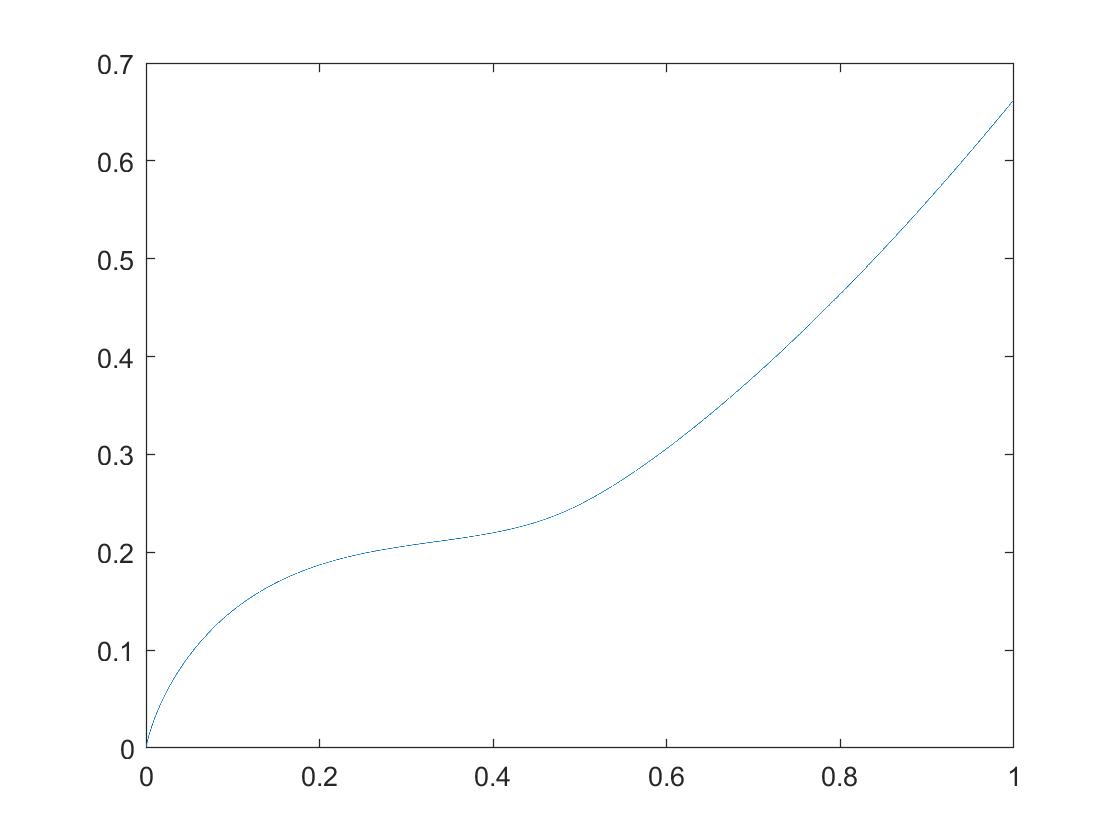}}

\caption{The graph of $g_{4}(\sigma)$}\label{g4}
\end{figure}
\section{Further remarks}
In \cite{MR4064778}, Sah et al$.$ obtained an inequality between $\ell_p$-spherical codes for different $p$; that is, $A_p(n, d)\leq A_q(n, d^{p/q})$ for all $1\leq q\leq p$ and $d\in(0, 1]$. So
\begin{equation}\label{12zhijian}
A_2(n, d)\leq A_p(n, d^{2/p}), \text{ if }1\leq p \leq 2,
\end{equation}
and
\begin{equation}\label{dayu2}
A_p(n, d)\leq A_2(n, d^{p/2}), \text{ if }p\geq2.
\end{equation}
Sah et al$.$ used inequality (\ref{dayu2}) to obtain an upper bound for $A_p(n, d)$ $(p\geq2)$.

On the other hand, Swanepoel \cite{MR1750139} had used inequality (\ref{12zhijian}) to obtain a lower bound for $A_p(n, 1/2)$ $(1.62107<p\leq2)$ before. Because the best lower bound for $A_2(n, d)$ has been improved since then, we update this type of lower bound here. We need the following theorem, which is the best known lower bound for $A_2(n, d)$ $(d\in(0, 1))$.
\begin{theorem}[\cite{2021arXiv211101255G}]
Let $\theta\in(0, \pi/2)$ be fixed. Then
$$
A_2(n, \sin(\theta/2))\geq(1+o(1))\ln\frac{\sin\theta}{\sqrt2\sin(\theta/2)}\cdot n\cdot\frac{\sqrt{2\pi n}\cos\theta}{\sin^{n-1}\theta}.
$$
\end{theorem}
For $1<p\leq2$, we have
$$
A_p(n, 1/2)\geq A_2(n, (1/2)^{p/2}).
$$
Let $\sin(\theta/2)=2^{-p/2}$. Then $\cos(\theta/2)=\sqrt{1-2^{-p}}, \sin\theta=2^{1-p/2}\sqrt{1-2^{-p}}$, and $\cos\theta=1-2^{1-p}$.
So
\begin{equation}\label{lingyige}
\begin{split}
A_p(n, 1/2)&\geq A_2(n, (1/2)^{p/2})\\
&=A_2(n, \sin(\theta/2))\\
&\geq(1+o(1))\ln\sqrt{2-2^{1-p}}\cdot n\cdot\frac{\sqrt{2\pi n}(1-2^{1-p})}{(2^{1-p/2}\sqrt{1-2^{-p}})^{n-1}}.
\end{split}
\end{equation}
After some numerical calculations, when $p\in(1.9948, 2]$, the lower bound in inequality (\ref{lingyige}) is better than that in  inequality (\ref{apxiajie}).
\section*{Acknowledgements}
The authors would like to thank Professor Hong Liu for helpful comments on the manuscript. 

\bibliographystyle{abbrv}
\bibliography{kissing_number_REF}
\end{document}